\documentclass{amsart}

\newtheorem{thm}{Theorem}[section]
\newtheorem{prop}[thm]{Proposition}
\newtheorem{cor}[thm]{Corollary}
\newtheorem{lemma}[thm]{Lemma}

\theoremstyle{definition}
\newtheorem{dfn}[thm]{Definition}

\theoremstyle{remark}
\newtheorem{rem}[thm]{Remark}

\numberwithin{equation}{section}

\newcommand{\grad}{\ensuremath{\mathrm{grad}\ }}

\newcommand{\tub}{\ensuremath{\mathrm{Tub} }}
\newcommand{\F}{\ensuremath{\mathcal{F}}}
\newcommand{\singularF}{\ensuremath{\mathcal{X}_{F}}}
\newcommand{\rank}{\ensuremath{\mathrm{rank}\ }}

\newcommand{\id}{\ensuremath{\mathrm{id}}}

\newcommand{\RR}{\mathbb R}

\begin{document}

\title{Equifocality of a singular riemannian foliation}

 
\author{Marcos M. Alexandrino}

\author{Dirk T\"{o}ben}
 

\address{Marcos M. Alexandrino\\Instituto de Matem\'{a}tica e Estat\'{\i}stica\\
Universidade de S\~{a}o Paulo, Rua do Mat\~{a}o 1010,05508 090 S\~{a}o Paulo, Brazil}
\email{marcosmalex@yahoo.de}

\address{Dirk T\"oben\\
Mathematisches Institut, Universit\"at zu K\"oln, Weyertal 86-90, 50931 K\"oln, Germany}
\email{dtoeben@math.uni-koeln.de}

\thanks{The first author was  supported by CNPq and partially supported by FAPESP}


\date{24 th May, 2007.}

\begin{abstract}
A singular foliation on a complete riemannian manifold $M$ is said to be riemannian if each geodesic that is perpendicular at one point to a leaf remains perpendicular to every leaf it meets. We prove that the regular leaves are equifocal, i.e., the end point map of a normal foliated vector field has constant rank. This implies that we can reconstruct the singular foliation by taking all parallel submanifolds of a regular leaf with trivial holonomy. In addition,  the end point map of a normal foliated vector field on a leaf with trivial holonomy is a covering map.
These results generalize previous results of the authors on singular riemannian foliations with sections.
\end{abstract}
\maketitle

\section{Introduction}

In this section, we will recall some definitions and state our main results as Theorem \ref{thm-s.r.f.-equifocal} and Corollary \ref{cor}.

We start by recalling the definition of a singular riemannian foliation (see  the  book of   Molino \cite{Molino}).
\begin{dfn}[s.r.f]
 A partition $\F$ of a complete riemannian manifold $M$ by connected immersed submanifolds (the \emph{leaves}) is called a {\it singular riemannian foliation} (s.r.f. for short) if it verifies condition (1) and (2):

\begin{enumerate}
\item $\F$ is a {\it singular foliation},
i.e., the module $\singularF$ of smooth vector fields on $M$ that are tangent at each point to the corresponding leaf acts transitively on each leaf. In other words, for each leaf $L$ and each $v\in TL$ with footpoint $p,$ there is $X\in \singularF$ with $X(p)=v$.
\item  $\F$ \emph{transnormal}, i.e., every geodesic that is perpendicular at one point to a leaf remains perpendicular to every leaf it meets.
\end{enumerate}
\end{dfn}

Let $\F$ be a singular riemannian foliation on a complete riemannian manifold $M.$  A leaf $L$ of $\F$ (and each point in $L$) is called \emph{regular} if the dimension of $L$ is maximal, otherwise $L$ is called {\it singular}. 

Typical examples of s.r.f. are the partition by orbits of an isometric action, by leaf closures of a Riemannian foliation, examples constructed by suspension of homomorphisms (see  \cite{Alex2,Alex4}) and examples constructed by changes of metric and surgery (see \cite{AlexToeben}).

A particular class of s.r.f. are the one which admits sections, i.e., for each regular point $p$ the set $\Sigma:=\exp(\nu_{p}L_{p})$ is a complete immersed submanifold that meets each leaf orthogonally.

The concept of singular riemannian foliations with sections (s.r.f.s. for short) was introduced in \cite{Alex2} and continued to be studied by the authors  in \cite{Alex1,Alex3,Alex4,Toeben,Toeben2,AlexToeben}, by Lytchak and Thorbergsson in \cite{LytchakThorbergsson} and recently by Gorodski and the first author  in \cite{AlexGorodski}. 
In \cite{Boualem} Boualem dealt with a 
singular riemannian foliation $\F$ on a complete manifold $M$ 
such that the distribution of normal spaces of the regular leaves  
is integrable. It was proved in \cite{Alex4} that such an 
$\F$ must be a s.r.f.s. 

S.r.f.s. include the partitions by orbits of a polar action and the well known class of isoparametric foliations on space forms, some of them with inhomogeneous leaves. 

In \cite{TTh1}, Terng and Thorbergsson introduced the concept of 
equifocal submanifolds with flat sections in compact symmetric spaces in 
order to generalize the definition of isoparametric submanifolds 
in euclidean space.

A connected immersed submanifold $L$ of a complete riemannian manifold 
$M$ is called \emph{equifocal} if it satisfies the following 
conditions:
\begin{enumerate}
\item The normal bundle $\nu(L)$ is flat.
\item For each parallel normal field $\xi$ on a neighborhood $U \subset L$,  
the derivative of the map $\eta_{\xi}:U\to M$ defined 
by $\eta_{\xi}(x):=\exp_{x}(\xi)$ has constant rank.
\item $L$ has sections, i.e.~for each~$p\in L$, 
the set $\Sigma :=\exp_{p}(\nu_p L_{p})$, called section,  is a complete immersed totally 
geodesic submanifold. 
\end{enumerate}

There is almost an equivalence between the notions of a s.r.f.s. and equifocal submanifolds that is worked out in the authors' works  \cite{Alex2} and \cite{Toeben}.  

On the one hand it was proved that a closed embedded  equifocal submanifold induces a s.r.f.s. by taking all its parallel submanifolds (\cite{Toeben},  \cite{Alex3}) if and only if there is exactly one section through every regular value of the normal exponential map of the equifocal submanifold. The global structure inherent to a s.r.f.s. was then used to generalize some results known for isoparametric submanifolds in euclidean space.

On the other hand, it was proved in  \cite{Alex2} that the leaves of a s.r.f.s. are equifocal (see \cite{Toeben} for an alternative proof). In converse direction to above the equifocality of a s.r.f.s. is also a very important tool in the theory of s.r.f.s. For example,  it allows us to have a Slice Theorem, singular holonomy, Weyl pseudogroups, a relation of s.r.f.s. to transnormal maps and an extension of Weyl-invariant forms to basic forms.

 While the existence of sections has interesting structural implications it naturally restricts the number of cases that are covered. This can be best seen in the case of homogenous s.r.f., when comparing an arbitrary isometric actions with a polar action. The latter is best exemplified by the action of a compact Lie group on itself by conjugation.
In this paper we want to drop the condition on the existence of sections and prove that regular leaves of a s.r.f. are also equifocal.  In order to make this statement precise, we will drop the first and third condition in the definition of equifocal submanifold and  we will also need to change the concept of parallel normal fields to foliated vector fields. Note that the restriction $\F_r$ of $\F$ to the regular stratum of $M$ is a regular foliation. We recall that a vector field $\xi$ in the normal bundle of the foliation over an open subset $U$ in the regular stratum is called \emph{foliated}, if for each vector field $Y\in\singularF$ the Lie bracket $[\xi,Y]$ also belongs to $\singularF$. If we consider a local submersion $\pi$ which describes the plaques of $\F$ in a neighboorhood of a point of $L$, then a normal foliated vector field is a normal projectable/basic vector field with respect to $\pi.$
\begin{rem}\label{rem-foliated-vector-field}
A Bott or basic connection $\nabla$ of a foliation $\F$ is a connection of $TM$ with $\nabla_XY=[X,Y]^{\nu\F}$ whenever $X\in \singularF$ and $Y$ is vector field of the normal bundle $\nu\F$ of the foliation. Here the superscript $\nu\F$ denotes projection onto $\nu\F$. A foliated vector field clearly is parallel with respect to the Bott connection. This connection can be restricted to the normal bundle of a leaf.
\end{rem}

\begin{dfn}
\label{dfn-foliated-vector-field}
Let $L$ be a regular leaf of a s.r.f. A normal vector field along $L$ is said to be \emph{foliated}, if it is Bott-parallel, or in other words, if it is locally the restriction of a foliated vector field of $\F_r$ to a neighborhood $U\subset L$.
\end{dfn}

\begin{rem}
Note that if the s.r.f. admits sections then a normal foliated vector field is a parallel normal field along each regular leaf $L$ with respect to the induced Levi-Civita connection on $\nu L$ and vice versa. In other words in the case of sections the induced Levi-Civita connection is a Bott-connection.  
\end{rem}

We are finally ready to state our result precisely.

\begin{thm}
\label{thm-s.r.f.-equifocal}
Let $\F$ be a s.r.f. on a complete riemannian manifold $M$. Then for each regular point $p$ there exists a neigborhood $U$ of $p$ in $L_{p}$ such that
\begin{enumerate}
\item[(1)] For each   normal foliated vector field $\xi$  along $U$  the derivative of  the map    $\eta_{\xi}:U\rightarrow M,$ defined as $\eta_{\xi}(x):=\exp_{x}(\xi),$ has constant rank.
\item[(2)] $W:=\eta_{\xi}(U)$ is an open set of $L_{\eta_{\xi}(p)}$.
\end{enumerate}
\end{thm}

\begin{cor}
\label{cor}
Let $L_{p}$ be a regular  leaf with trivial holonomy and $\Xi$ denote the set of all  normal foliated vector  fields along $L_{p}.$ 
\begin{enumerate}
\item[(1)] Let $\xi\in \Xi$. Then $\eta_{\xi}:L_{p}\rightarrow L_{q}$ is a covering map if $q=\eta_{\xi}(p)$ is a regular point.
\item[(2)]  $\F=\{\eta_{\xi}(L_{p})\}_{\xi\in \, \Xi},$ i.e., we can reconstruct the singular foliation by taking all parallel submanifolds of the regular leaf $L_{p}.$ 
\end{enumerate}
\end{cor}

This paper is organized as follows. In Section \ref{sec-prop} we present the propositions need to prove the theorem. In particular we prove two propositions which contain some  improvements of  Molino's results on the local analysis of a s.r.f. More precisely, we review  a local decomposition result and a product theorem due to Molino (see Proposition \ref{lemma-almost-product} and Proposition \ref{lemma-metric-in-S}). In Section \ref{sec-teo} we prove Theorem \ref{thm-s.r.f.-equifocal} and in Section \ref{sec-cor} we prove Corollary \ref{cor}.

\section{Properties of a s.r.f.} 
\label{sec-prop}

In this section we will present the propositions needed to prove  Theorem \ref{thm-s.r.f.-equifocal}.  
Throughout this section we assume that $\F$ is a s.r.f. on a complete riemannian manifold $M.$

We start by recalling the so called \emph{Homothetic Transformation Lemma} of Molino (see Lemma 6.2 \cite{Molino}).

By conjugating the homothetic transformations of the normal bundle of a plaque $P$  via the normal exponential map, one defines for small strictly positive real numbers $\lambda$, a homothetic transformation $h_{\lambda}$  with proportionality constant $\lambda$ with respect to the plaque $P.$

\begin{prop}[\cite{Molino}]
\label{homothetic-lemma}
The homothetic transformation $h_{\lambda}$ sends plaque to plaque and therefore respects the singular foliation $\F$ in the tubular neigborhood $\tub(P)$ where it is defined.

\end{prop}

The next two propositions contain some  improvements of  Molino's results (compare with Theorem 6.1 and Proposition 6.5 of \cite{Molino}).


\begin{prop}
\label{lemma-almost-product}
\
Let $g$ be the original metric on $M$ and $q\in M.$  Then there exists a tubular neighborhood $\tub(P_{q})$ and  a new metric $\tilde{g}$ on $\tub(P_{q})$ with the following properties.
\begin{enumerate}
\item[(a)] For each $x\in \tub(P_{q})$ the normal space of the leaf $L_x$ is tangent to the slice $S_{\tilde{q}}$ which contains $x$, where $\tilde q\in P_q.$ 
\item[(b)] Let $\pi:\tub(P_{q})\rightarrow P_q$ be the orthogonal projection. Then the restriction $\pi|_{P_{x}}$ is a riemannian submersion.
\item[(c)] $\F\cap \tub(P_{q})$ is a s.r.f.
\item[(d)] $\F\cap S_{\tilde{q}}$ is a s.r.f. for each $\tilde{q}\in P_{q}.$
\item[(e)] The associated transverse metric is not changed,i.e., the distance between the plaques with respect to $g$ is the same distance between the plaques with respect to $\tilde{g}$.
\item[(f)] If a curve $\gamma$ is a geodesic orthogonal to $P_{q}$ with respect to the original metric $g$, then 
$\gamma$ is a geodesic orthogonal to $P_{q}$ with respect to the new metric $\tilde{g}$.
\end{enumerate}
\end{prop}
\begin{proof}

Let $X_1,\ldots, X_r \in  \singularF$ (i.e. vector fields that are always tangent to the leaves) 
so that $\{X_i(q)\}_{i=1,\ldots ,r}$ is a linear basis of $T_{q}P_{q}$.  Let $\varphi_{t_1}^1,\ldots, \varphi_{t_r}^r$ denote the associated one parameter groups and define $\varphi(t_1,\ldots,t_r,y):= \varphi_{t_1}^1\circ\cdots\circ\varphi_{t_r}^r$ where $y\in S_{q}$ and $(t_1,\ldots,t_r)$ belongs to a neighborhood $U$  of $0\in\mathbb{R}^r.$  Then, reducing $U$ and $\tub(P_{q})$ if necessary, one can guarantee the existence of   a regular foliation $\F^{2}$ with plaques   $P^{2}_{y}= \varphi(U,y).$ We note that the plaques $P^{2}_{z}\subset P_{z}$ and each plaque $P^{2}$ cuts each slice  at exactly one point.  Using the fact that  $\pi|_{P^{2}_{y}}: P^{2}_{y}\rightarrow P_{q}$ is a diffeomorphism, we can define a metric on each plaque $P^{2}_{y}$ as
$\tilde{g}^{2}:=(\pi|_{P^{2}_{y}})^{*}g.$

Now we want to define a metric $\tilde{g}^1$ on each slice $S\in \{S_{\tilde{q}}\}_{\tilde{q}\in P_{q}} .$ Set $D_{p}:=\nu_{p} L^{2}_{p}$ and define $\Pi:T_{p}M\rightarrow D_{p}$ as the orthogonal projection with respect to $g$. The fact that each plaque $P^{2}$ cuts each slice  at  one point implies that $\Pi|_{T_{p}S}:T_{p}S\rightarrow D_{p}$ is an isomorphism. Finally we define $\tilde{g}^{1}:= (\Pi|_{T_{p}S})^{*}g$ and $\tilde{g}:=\tilde{g}^{1}+\tilde{g}^{2}$, meaning that $\F^2$ and the slices meet orthogonally. Items (a) and (b) follow directly from the definition of $\tilde{g}.$ 

To prove Item (c) it suffices to prove that the plaques of $\F$ are locally equidistant to each other. Let $x\in S_{\tilde{q}}$, $P_{x}$ a plaque of $\F$. We know that the plaques of $\F$ are contained in the leaves of the foliation by distance-cylinders $\{C\}$ with axis $P_{x}$ with respect to $g$.
We will prove that each $C$ is also a distance-cylinder with axis $P_{x}$ with respect to the new metric $\tilde{g}.$ These facts and  the arbitrary choice of $x$  will imply that the plaques of $\F$ are locally equidistant to each other.

First we recall that a smooth function $f:M\rightarrow \mathbb{R}$ is called a \emph{transnormal function} with respect to the metric $g$ if  there exists a $C^{2}(f(M))$ function  $b$  such that $g(\grad f,\grad f)=b\circ f$. 
Let $f:\tub(P_{x})\rightarrow \mathbb{R}$ be a smooth transnormal function with respect to the metric $g$ so that each regular level set $f^{-1}(c)$ is a cylinder $C$ with axis $P_x$,  e.g. $f(y)=d(y,P_x)^2$. 
Let $\widetilde{\grad} f$ denote the gradient of $f$ with respect to the metric $\tilde{g}.$ It follows from the construction of $\tilde{g}$ that 
 \begin{equation}
 \label{eq-widetildegrad-grad}
 \widetilde{\grad} f=\grad f+l 
 \end{equation}
 where $l$ is a vector tangent to a plaque of $\F^{2}$ and in particular to a plaque of $\F$. 

Indeed, let $v\in D_p$ and $w:=(\Pi|_{T_{p}S})^{-1}(v)$. Then 
\begin{eqnarray*}
  g(\grad f,v) &=& d f (v) \\
               &=& d f (w)\\
               &=& \tilde{g}(\widetilde{\grad} f, w)\\
	          &=& \tilde{g}^{1}(\widetilde{\grad} f, w)\\
               &=& g(\Pi\widetilde{\grad} f,\Pi w)\\
               &=&g(\Pi \widetilde{\grad}f, v) 
\end{eqnarray*}
We conclude from the arbitrary choice of  $v\in D_p$, that  $\grad f=\Pi \widetilde{\grad} f,$  and hence  
$ \widetilde{\grad} f=\grad f+l$.

Equation \ref{eq-widetildegrad-grad} implies that $f$ is a also a transnormal function with respect to the metric $\tilde{g}$, i.e., 
\begin{equation}
\label{f-transnormal-g-tildeg}
\tilde{g}(\widetilde{\grad} f,\widetilde{\grad} f)=b\circ f, 
\end{equation}
Indeed, 
\begin{eqnarray*}
\tilde{g}(\widetilde{\grad} f,\widetilde{\grad} f)&=& d f(\widetilde{\grad} f)\\
                                                  &=& d f(\grad f)\\
                                                  &=& g(\grad f, \grad f)\\
                                                  &=& b\circ f
\end{eqnarray*}

Using a local version of Q-M Wang's theorem \cite{Wang}, we conclude that each regular level set of $f$ (i.e. $C$ ) is a distance cylinder around $P_{x}$ with respect to the metric $\tilde{g}$.

Now we want to prove Item (d).  Set $P_{x}^{s}=P_{x}\cap S_{\tilde{q}}$ and  $C^{s}:=C\cap S_{\tilde{q}}.$ It suffices to note that the singular foliation $\{C^s\}$ is a foliation by cylinders with axis $P_{x}^s$ with respect to the new metric $\tilde{g}.$ This follows from the fact that $\nu_{x}P_{x}\subset T_{x}S_{\tilde{q}}$ and that each geodesic orthogonal to $P_{x}$ at $x$ is contained in $S_{\tilde{q}}$ (see Item (a)).

In particular we conclude that the distance between $C$ and $P_x$ and the distance between $C^s$ and $P_{x}^{s}$ with respect to the metric $\tilde{g}$ are the same. 

To prove Item (e) we have to prove that the distance between the cylinder $C$ and the plaque $P_{x}$ is the same for both metrics.
Let $f$ be the transnormal function (with respect to $g$) defined above. 
According to Q-M Wang \cite{Wang}  for $k=f(P_x)$ and a regular value $c$  we have  
$d(P_{x},f^{-1}(c))=\int_{c}^{k}\frac{ds}{\sqrt{b(s)}}.$ Since $f$ is also a transnormal function with respect to $\tilde{g}$ (see Equation (\ref{f-transnormal-g-tildeg})), we conclude that $d(P_{x},C)=\tilde{d}(P_{x},C),$
for  $C=f^{-1}(c).$

Finally we prove Item (f). We consider the transnormal function $f$ above with $x=q$. 
In this case,  
Equation (\ref{eq-widetildegrad-grad}) and the fact that $\grad f\in D_p\cap T_pS$ imply that $\grad f=\widetilde\grad f$. On the other hand, the integral curves of the gradient of a transnormal function are geodesic segments up to reparametrization (see e.g. \cite{Wang}). Therefore the radial geodesics of $P_{q}$ coincide in both metrics. This finishes the proof.


\end{proof}

\begin{prop}
\label{lemma-metric-in-S}
Let $\tilde{g}$ be the metric  defined in Proposition \ref{lemma-almost-product}. Then there exists  a new metric $g_0$  on $\tub(P_q)$ so that,
\begin{enumerate}
\item[(a)] Consider the tangent space $T_{\tilde{q}}S_{\tilde{q}}$ with the metric $\tilde{g}$ and $S_{\tilde{q}}$ with the metric $g_0$. Then  $\exp_{\tilde{q}}:T_{\tilde{q}}S_{\tilde{q}}\rightarrow S_{\tilde{q}}$ is an isometry.
\item[(b)] For this new metric $g_0$ we have that $\F\cap S_{\tilde{q}}$ and $\F$ restricted to $\tub(P_q)$ are also s.r.f. 
\item[(c)]  For each $x\in \tub(P_{q})$ the normal space of the leaf $L_x$ is tangent to the slice $S_{\tilde{q}}$ which contains $x$, where $\tilde q\in P_q.$
\end{enumerate}

\end{prop}

\begin{rem}
\label{rem-prop-flat-metric-in-S}
Clearly a curve $\gamma$ which is a geodesic orthogonal to $P_{q}$ with respect to the original metric, remains a  geodesic orthogonal to $P_{q}$ with respect to the new metric $g_0$.
\end{rem}
\begin{proof}

Let  $\Pi_1$ be the orthogonal projection to the slices, recall that  $\tilde{g}^{1}= \tilde{g}\circ\Pi_1$ and $\tilde{g}^{2}=\tilde{g}\circ d\,\pi$.  Let  $h_{\lambda}$ denote  the homothetic transformation with respect to $P_q$.
Define $g_{\lambda}=\frac{1}{\lambda^2}h^{*}_{\lambda} \tilde{g}^{1}+\tilde{g}^{2}.$ Note that  the metric $g_{\lambda}$ tends uniformly to  a metric $g_0$  for $\lambda\to 0$. 
This metric $g_0$ restricted to $S_{\tilde{q}}$ is 
the induced metric on $\nu P_{\tilde{q}}$, where $\tilde{q}\in P_{q}.$ 

This implies that $\mathrm{L}_{\lambda}$ tends uniformly to $\mathrm{L}_{0}$, where $\mathrm{L}_{\lambda}$ is the length function. It follows then that
\begin{equation}
\label{lemma-metric-in-S-Eq1}
\lim_{\lambda\rightarrow 0} d_{\lambda}(x,P)= d_{0}(x,P)
\end{equation}
where  $P$ is a plaque.

Now we claim that $\F$ is a s.r.f. with respect to $g_{\lambda}$. Indeed, since $h^{*}_{\lambda} \tilde{g}^{2}=\tilde{g}^{2}$ and the homothetic transformation $h_{\lambda}$ sends plaque to plaque (see Proposition \ref{homothetic-lemma}) it suffices to prove that $\F$ is a s.r.f. with respect to $\frac{1}{\lambda^2} \tilde{g}^{1}+\tilde{g}^{2}$ . 
Let $f:\tub(P_{x})\rightarrow \mathbb{R}$ be a smooth transnormal function with respect to the metric $\tilde{g}$ so that each regular level set $f^{-1}(c)$ is a cylinder with axis $P_x$. Note that $f$ is also a transnormal function with respect to the metric  $\frac{1}{\lambda^2} \tilde{g}^{1}+\tilde{g}^{2},$ because $\widetilde{\grad} f$ is tangent to the slice. 
Using a local version of Q-M Wang's theorem \cite{Wang}, we conclude that each regular level set of $f$  is a tube over $P_{x}$ with respect to the metric $\frac{1}{\lambda^2} \tilde{g}^{1}+\tilde{g}^{2}$ . Therefore the plaques are equidistant to $P_x$ and hence we conclude that $\F$ is a s.r.f. with respect to $\frac{1}{\lambda^2} \tilde{g}^{1}+\tilde{g}^{2}$ . 

Finally let $x$ and $y$ be points which belong to the same plaque. Using Equation \ref{lemma-metric-in-S-Eq1} and the fact that $\F$ is a s.r.f. with respect to $g_{\lambda}$ we conclude that 

\begin{eqnarray*}
0&=&\lim_{\lambda\rightarrow 0} (\, d_{\lambda}(x,P)-d_{\lambda}(y,P)\,)\\
          &=& d_{0}(x,P)-d_{0}(y,P) 
\end{eqnarray*}

The above equation implies  that the plaques are locally equidistant and hence that the singular foliation $\F$ is riemannian.


Now we want to prove Item (c). Let $P_x$ be a plaque with $x\in S$. Note that for each metric $g_{\lambda}$ the normal space $H_{\lambda}$ of $P_x$ at $x$ (with respect to the metric $g_{\lambda}$) is tangent to $S.$ This fact will imply that the normal space of $P_{x}$ at $x$ with respect to $g_{0}$ is also tangent to $S.$ Indeed, we can find a sequence of normal spaces $H_{1/n}$ such that $H_{1/n}$ converge to a subspace $H_0$ tangent to $S$ at $x.$ Then we can find a  subsequence of frames $\{e_{i}^{n}\}$ which converge to a frame $\{e_{i}\}$ such that $\{e_{i}^{n}\}$ and $\{e_{i}\}$ are  bases of $H_{1/n}$ and $H_0$ respectively. 
Since 
\[\frac{1}{\lambda^2}h^{*}_{\lambda} \tilde{g}^{1}(d(\exp_{q})_{V}Y,d(\exp_{q})_{V}Z )=\tilde{g}^{1}(d(\exp_{q})_{\lambda V}Y,d(\exp_{q})_{\lambda V}Z),\]
we conclude that
\[g_{0}(e_{i},l)=\lim_{n\rightarrow\infty} g_{1/n}(e_{i}^{n}, l)=0\]
where $l$ is tangent to the plaque.
The last equation implies that $H_0$ is the normal space of $P_{x}$ at $x$ with respect to $g_0$.  

\end{proof}


\begin{prop}
\label{lemma-slice-fundamental}
Let $S_{q}$ be a slice at $q$ and
  $\varphi:S_{q}\rightarrow S_{q}$ be the geodesic symmetry at $q,$ i.e., $\varphi=\exp_{q}\circ(-\id)\circ\exp_{q}^{-1}$. Then 
the map $\varphi$ is $\F\cap S_{q}$ foliated, i.e. the foliation $\F\cap S_{q}$ is invariant by the involution $\varphi$.
\end{prop}

\begin{proof}

It follows from Proposition \ref{lemma-metric-in-S} and Remark \ref{rem-prop-flat-metric-in-S} that  we can lift $\F$ via the exponential map in a neighborhood of $q$ to a s.r.f. of $T_{q}S_{q}$. Therefore we assume that $\F$ is a s.r.f. of $\RR^n$ with euclidean metric which has $\{0\}$ as a leaf.

\begin{lemma}\label{lemma:indSRF}
The induced singular foliation on the unit sphere $\F':=\F|S^{n-1}$ is a s.r.f.
\end{lemma}
\begin{proof}
First note that every leaf of $\F$ that has a point in $S^{n-1}$ lies in $S^{n-1}$. Clearly $\F'$ is a singular foliation. We now want to show transnormality. Let $v\in S^{n-1}$ and $\xi\in\nu_vL_v\cap T_vS^{n-1}$ a unit vector. We denote by $\gamma_\xi$ the geodesic in $S^{n-1}$ with initial vector $\xi$. We want to show that $\xi(t):=\dot\gamma_\xi(t)\in \nu_{w}L_w$, where $w=\gamma_\xi(t)$. First we assume $t\in (0,\pi)$. Then the two unit radial vectors of $S^{n-1}$ in $v$ and $w$ span a 2-plane of $\RR^n$ containing the origin. As it contains the straight line from $v$ to $w$, it lies in $\nu_wL_w$ by transnormality of $\F$. The intersection of this 2-plane with $S^{n-1}$ is exactly the geodesic $\gamma_\xi$. Therefore $\xi(t)\in \nu_wL_w$. This shows that $\gamma_\xi|[0,t)$ and consequently $\gamma_\xi|[0,\pi)$ is transnormal. To prove transnormality of $\gamma_\xi|[0,2\pi)$ we repeat the argument with $w$ respectively $\xi(t)$ as our new $v$ respectively $\xi$. Since the geodesic $\gamma_\xi$ is closed of period $2\pi$ only a third step is needed to show its transnormality.
\end{proof}

Now let $v\in S^{n-1}$ and let $L_v$ be leaf through $v$. Here we denote by $\nu_vL_v$ the normal space of $L_v$ in $S^{n-1}$. For any $\xi\in\nu^1_vL_v$ the geodesic $\gamma_\xi$ in $S^{n-1}$ meets the leaf $L_{-v}$ in the antipodal point $-v$ orthogonally, i.e. $-\xi=\dot\gamma_\xi(\pi)\in\nu_{-v}L_{-v}$. So as vector spaces in $\RR^n$ we have $\nu_{v}L_v\subset \nu_{-v}L_{-v}$ and by symmetry we have equality for every $v\in S^{n-1}$. In other words $-\id$ respects the normal bundle and therefore also the tangent bundle of $\F$. From this we conclude that $-\id$ respects $\F$ on $S^{n-1}$.
\end{proof}

\begin{cor}
\label{singular-points-isolated}
Let $\gamma$ be a geodesic orthogonal to a regular leaf of a s.r.f. Then the singular points are isolated on $\gamma.$ 
\end{cor}

\section{Proof of the theorem}
\label{sec-teo}

In this section we will apply the above propositions to prove the theorem. We start by proving a local version of Theorem \ref{thm-s.r.f.-equifocal}.

\begin{prop}
\label{prop-s.r.f-equifocal-tubular}
Let $\tub(P_{q})$ be a tubular neighborhood of a plaque $P_{q},$
$x_{0}\in \tub(P_{q}),$  a regular point  and  $\xi\in \nu P_{x_{0}}$ such that  $\exp_{x_{0}}(\xi)=q.$  Then we can find a neighborhood $U$ of $x_{0}$ in $P_{x_{0}}$ with the following properties:
\begin{enumerate}
\item[1)] We can extend $\xi$ to a foliated normal vector field $\xi$ on $U.$ 
 \item[2)] The geodesic segment that is orthogonal to $P_{q}$ and contains  a point $x\in U$ is $\gamma_{x}(t):=\exp_{x}((t+1)\,\xi)$ where $t\in[-1,1].$
\item[3)] $\eta_{(t+1)\,\xi}(U)$ is an open subset of $L_{\gamma_{x_{0}}(t)}.$ 
\item[4)] $\eta_{t\,\xi}:U\rightarrow \eta_{t\,\xi}(U)$ is a  diffeomorphism  for $t\neq 1.$
\item[5)] $\dim \rank D \eta_{\xi}$ is constant on $U.$ 
\end{enumerate}
\end{prop}

\begin{proof}

The proof of 1) is straightforward. The proof of 2) follows from the Homothetic Transformation Lemma by Molino (Proposition \ref{homothetic-lemma}).

Using Proposition \ref{homothetic-lemma} and Proposition \ref{lemma-slice-fundamental} we can conclude the following lemma.
\begin{lemma}
\label{lemma-jacobi-field-itens}
Let  $\alpha(s)$ be a curve in $U.$ Define $f(s,t)=\exp_{\alpha(s)}(t\xi)$ and $J(t)=\frac{\partial f}{\partial s}(0,t).$ Then to prove item 3), 4) and 5) it suffices to prove that the Jacobi field $J$ is always tangent to the leaves.
\end{lemma}

In what follows we will prove that the Jacobi field $J$ defined above is always tangent to the leaves.

Let $g_{0}$ be the metric defined in Proposition \ref{lemma-metric-in-S}. 
Then Remark \ref{rem-prop-flat-metric-in-S} and Item 2)  
 imply that the Jacobi field $J$ defined in Lemma \ref{lemma-jacobi-field-itens} has not been changed when the metric was modified. 

Now consider a geodesic segment $\gamma$ orthogonal to the leaves of $\F$ so that $\gamma(0)=q$ and $\gamma(1)$ is a regular point contained in $S_q$ . It follows from 
Corollary \ref{singular-points-isolated} that $\gamma(t)$ is always regular for $ -1\leq t<0$ and $0<t\leq 1.$
 
We define $\sigma$ as the submanifold contained in $S_{q}$ which is the image by $\exp_q$ of a subspace and so that $\sigma$ is orthogonal to $L_x$ at $x.$  

By Proposition \ref{lemma-slice-fundamental},  Proposition \ref{lemma-metric-in-S} and Proposition \ref{homothetic-lemma}  we have that the plaques $P_{\gamma(t)}\cap S_{q}$ are orthogonal to $\sigma$ for $-1\leq t\leq 1$. Then it follows from Proposition \ref{lemma-metric-in-S} that the plaques $P_{\gamma(t)}$ are orthogonal to $\sigma$ for $-1\leq t\leq 1.$

Consider a geodesic segment $\beta$ so that $\beta(0)=\gamma(t)$ and $\beta$ is orthogonal to $P_{\gamma(t)}$. Then  Proposition \ref{lemma-metric-in-S} imply that $\beta$ is contained in $S_{q}.$ Since $S_q$ is identified with $T_{q}S_{q}$  we can consider $\beta$ as a straight line. Since $P_{\gamma(t)}\cap S_{q}$ is orthogonal to $\sigma$, and $\sigma$ is identified with a subspace, we conclude that $\beta$ is contained in $\sigma$.

Therefore $\exp_{\gamma(t)}(\nu(P)_{\gamma(t)}\cap B_{\epsilon}(0))$  is an open set of $\sigma.$ A standard argument from riemannian geometry implies that the second form is null at $\gamma(t)$, i.e., $\sigma$ is geodesic at $\gamma(t).$ In particular the curvature tensor $R$ of $\sigma$ is the same as the ambient space at $\gamma(t).$ This fact and the fact that $R(\gamma^{'},\cdot)\gamma^{'}$ is self-adjoint imply that 
$T_{\gamma(t)}\sigma$ as well $(T_{\gamma(t)}\sigma)^{\perp}$ are families of parallel subspace along $\gamma$ which are invariant by $R(\gamma^{'},\cdot)\gamma^{'}.$

Finally consider the $L_x$ -Jacobi field $J$ defined in Lemma \ref{lemma-jacobi-field-itens}. This Jacobi field has initial conditions at  $(T_{\gamma(1)}\sigma)^{\perp}$ and satisfies
the Jacobi equation. So $J(t)\in (T_{\gamma(t)}\sigma)^{\perp}$ for $-1\leq t\leq 1.$

As remarked above plaques $P_{\gamma(t)}$ are orthogonal to $\sigma$ for $-1\leq t\leq 1.$ Since $P_{\gamma(t)}$ are regular plaques for $t\neq0$ (see Corollary  \ref{singular-points-isolated}) we conclude that $J(t)$ is always tangent to $P_{\gamma(t)}.$

\end{proof}

We are finally ready to prove   Theorem \ref{thm-s.r.f.-equifocal}. 

Let $L$ be a leaf of $\F,$  and $\xi$ be a  normal foliated vector field along a neighborhood $U$ of $L.$ 
Let $p\in U.$ Since singular points are isolated along 
   $\gamma_{p}(t)=\exp_{p}(t\,\xi)|_{[-\epsilon,1+\epsilon]}$, there exists a partition $0=t_{0}<\cdots<t_{n}=1$ such that $\gamma(t_{i})$ are the only possible singular points.  
   

Let  $P_{\gamma_{p}(r_{i})}$ be regular plaques that belong to  $\tub(P_{\gamma_{p}(t_{i-1})})\cap \tub(P_{\gamma_{p}(t_{i})}),$ where $t_{i-1}<r_{i}<t_{i}.$ Applying  Proposition  \ref{prop-s.r.f-equifocal-tubular}  we can find  an open set $U_{0}\subset P_{p},$ of the plaque $P_{p},$   an open set $U_{n+1}$ of $P_{\gamma_{p}(1)},$ open sets $U_{i}\subset P_{\gamma_{p}(r_{i})}$ of the plaques $P_{\gamma_{p}(r_{i})}$ (for $1\leq i\leq n)$ and   normal foliated vector fields  $\xi_{i}$ along  $U_{i},$ (for $0\leq i\leq n$) with the following properties: 
\begin{enumerate}
\item[1)] For each  $U_{i},$ the  normal foliated vector field  $\xi_{i}$  is tangent to the geodesics $\gamma_{x}(t),$ where $x\in U_{0};$
\item[2)] $ \eta_{\xi_{i}}:U_{i}\rightarrow U_{i+1}$ is surjetive and for $i< n$ a   diffeomorphism.
\item[3)] $\eta_{\xi}|_{U_{0}}=\eta_{\xi_{n}}\circ\eta_{\xi_{n-1}}\circ\cdots\circ\eta_{\xi_{0}}|_{U_{0}}$ 
\end{enumerate}
   
Because $\dim\rank d\eta_{\xi_{i}}$ is constant on $U_{i}$, it follows that   $\dim d\eta_{\xi}$ is constant on $U_{0}.$ Since this hold for each $p\in U,$ 
$\dim d\eta_{\xi}$ is constant on $U.$ It also follows that 
$\eta_{\xi}(U)$ is an open set of $L_{\eta_{\xi}(U)}$.


\section{Proof of Corollary \ref{cor}}
\label{sec-cor}

Let $L_{p}$ be a regular  leaf with trivial holonomy and $\xi$ a   normal foliated vector  fields along $L_{p}$. 
It follows from Theorem \ref{thm-s.r.f.-equifocal} that $\eta_{\xi}(L_{p})$ is an open set of $L_{q},$ where $q=\eta_{\xi}(p).$ 
In this section we  will prove that $\eta_{\xi}(L_{p})$ is also a closed set in $L_{q}$ and hence that $\eta_{\xi}(L_{p})=L_{q}.$ In addition, when $q$ is a regular point, we will also prove that 
 $\eta_{\xi}:L_{p}\rightarrow L_{q}$ is a covering map.

At first suppose  that $L_{q}$ is a regular leaf. 

For a point $z_{0}\in L_{q}$ assume that there exists a point $z_1\in\eta_{\xi}(L_{p})$ which also belongs to the plaque $P_{z_{0}}.$ 
Let $x_{\alpha}$ be a point in  $L_{p}$ such that $\eta_{\xi}(x_{\alpha})=z_{1}.$   
Let $\hat{\xi}_{\alpha}$ be the vector in $T_{z_{1}}M$ tangent to the geodesic $\exp_{x_{\alpha}}(t\,\xi)$ so that $\exp_{z_{1}}(\hat{\xi}_{\alpha})=x_{\alpha}.$
We can extend $\hat{\xi}_{\alpha}$ along  the plaque $P_{z_{0}}.$ Theorem \ref{thm-s.r.f.-equifocal} implies that $\eta_{\hat{\xi}_{\alpha}}:P_{z_{0}}\rightarrow L_{p}.$ 
Let $A$ be the set of points $ z \in P_{z_{0}}$ such that  $\hat{\xi}_{\alpha}(z)$ is  tangent to the geodesic $\exp_{x}(t\,\xi)$ and $\exp_{z}(\hat{\xi}_{\alpha})=x$ for $x\in L_{p}$.
The fact $\eta_{\xi}:L_{p}\rightarrow L_{q}$ is a local diffeomorphism (see  Theorem  \ref{thm-s.r.f.-equifocal}) implies that $A$ is an open set of $P_{z_{0}}.$ On the other hand, the fact that $\eta_{\hat{\xi}_{\alpha}}:P_{z_{0}}\rightarrow L_{p}$ is a local diffeomorphism  implies that $A$ is a closed set of $P_{z_{0}}.$ Therefore $A=P_{z_{0}}.$  This means that $z_{0}\in \eta_{\xi}(L_{p})$ and hence that $\eta_{\xi}(L_{p})$ is  a closed set in $L_{q}.$

Now we want to prove that $\eta_{\xi}:L_{p}\rightarrow L_{q}$ is a covering map, for a regular point $q$. 
For a plaque $P_{z}$ consider all  points $x_{\alpha}\in L$ so that $\eta_{\xi}(x_{\alpha})=z.$ 
For each $x_{\alpha}$ let
 $\hat{\xi}_{\alpha}$ be the vector in $T_{z}M$ tangent to the geodesic $\exp_{x_{\alpha}}(t\,\xi)$ so that $\exp_{z}(\hat{\xi}_{\alpha})=x_{\alpha}.$
As proved above, we can extend each  vector $\hat{\xi}_{\alpha}$ to a vector field along the plaque $P_{z}$ and we can show that the map $\eta_{\xi}:W_{\alpha}\rightarrow P_{z}$ is a diffeomorphism, where $W_{\alpha}=\eta_{\hat{\xi}_{\alpha}}(P_{z})$. Note that $\eta_{\xi}^{-1}(P_{z})=\cup_{\alpha} W_{\alpha}.$ We conclude   that $\eta_{\xi}:L_{p}\rightarrow L_{q}$ is a covering map.

At last, suppose that $L_{q}$ is a singular leaf. 

For a point $z_{0}\in L_{q}$ assume that there exists a point $z_1\in\eta_{\xi}(L_{p})$ which also belongs to the plaque $P_{z_{0}}.$ 
There exists $x_{1}\in L_{p}$ such that $z_{1}=\eta_{\xi}(x_{1})\in P_{z_{0}}.$ We can find a $s<1$ such that $y_{1}=\eta_{s\,\xi}(x_{1})$ is a regular point. Since $y_{1}$ is a regular point, the plaque $P_{y_{1}}$ is an open set of $\eta_{s\,\xi}(L_{p})$. There exists a parallel normal field $\hat{\xi}$ along $P_{y_{1}}$ such that $\eta_{\hat{\xi}}\circ\eta_{s\,\xi}=\eta_{\xi}.$ 
 It follows  that $\eta_{\hat{\xi}}(P_{y_{1}})\subset P_{z_{0}}.$ On the other hand, since the foliation is singular,  the plaque $P_{y_{1}}$ intersect the slice $S_{z_{0}}.$ These two facts  imply that $z_{0}\in  \eta_{\hat{\xi}}(P_{y_{1}}).$ Therefore $z_{0}\in \eta_{\xi}(L_{p})$ and hence $\eta_{\xi}(L_{p})$ is  a closed set in $L_{q}.$

\bibliographystyle{amsplain}

\end{document}